\documentclass{amsart}
\numberwithin{equation}{section}

\usepackage{mathtools,amsmath,amsthm,amsfonts,amssymb,enumerate}
\usepackage[T1]{fontenc}
\usepackage{hyperref}
\usepackage{palatino}
\usepackage{url}
\usepackage{euscript}
\usepackage{xcolor}
\usepackage{enumitem}

\makeatletter
\def\mathcolor#1#{\@mathcolor{#1}}
\def\@mathcolor#1#2#3{%
  \protect\leavevmode
  \begingroup
    \color#1{#2}#3%
  \endgroup
}
\makeatother

\usepackage{pstricks}
\usepackage[all]{xy}
\SelectTips{cm}{}
\usepackage{pinlabel}

\newcommand{\ig}[2]{\vcenter{\xy (0,0)*{\includegraphics[scale=#1]{fig/#2}} \endxy}}

\newtheorem{thm}{Theorem}[section]
\newtheorem{lem}[thm]{Lemma}

\newtheorem{prop}[thm]{Proposition}
\newtheorem{cor}[thm]{Corollary}

\theoremstyle{definition}
\newtheorem{defn}[thm]{Definition}
\newtheorem{nota}[thm]{Notation}
\newtheorem{exa}[thm]{Example}

\theoremstyle{remark}
\newtheorem{rem}[thm]{Remark}

\newcommand{\nc}{\newcommand} 
\nc{\mb}{\mathbb}
\nc{\mc}{\mathcal}
\nc{\mf}{\mathfrak}

\newcommand{\co}{\colon}
\nc{\ic}{\mathbf{IC}}

    \def\FM{{\mathbb{F}}}

    \def\NM{{\mathbb{N}}}

    \def\ZM{{\mathbb{Z}}}

    \def\EC{{\mathcal{E}}}

\def\HB{{\mathbf H}}    \def\HC{{\mathcal{H}}}

    \def\LC{{\mathcal{L}}}

\def\NB{{\mathbb N}}

\def\QB{{\mathbb{Q}}}    
\def\RB{{\mathbb R}}    
    
    \def\TC{{\mathcal{T}}}

\def\ZB{{\mathbb Z}}    

\def\FS{{\EuScript F}}

\def\ZS{{\EuScript Z}}

\usepackage{mathrsfs}

\def\d{\delta}

\def\l{\lambda}

\usepackage{bbm}

\def\F{{\mathbbm F}}

\def\1{\mathbbm{1}}

\newcommand{\Hom}{{\rm Hom}}
\newcommand{\JW}{{\rm JW}}

\newcommand{\End}{{\rm End}}

\newcommand{\TL}{\TC\LC}

\newcommand{\on}{\operatorname}
\newcommand{\ol}{\overline}
\newcommand{\ul}{\underline}

\newcommand{\sa}{\mathcolor{blue}{s}}
\newcommand{\tr}{\mathcolor{red}{t}}

\DeclareMathOperator{\spn}{span}

\DeclareMathOperator{\supp}{\mathrm{supp}}

\newcommand{\Zvv}{\ZB[v^{\pm 1}]}

\newcommand{\sumset}{\stackrel{\scriptstyle{\oplus}}{\scriptstyle{\subset}}}

\newcommand{\SD}{\mathbb S \mathcal D}

\usepackage{bbm}

\def\1{\mathbbm{1}}

\title{$p$-Jones-Wenzl idempotents}

\author{Gaston Burrull, Nicolas Libedinsky, and Paolo Sentinelli}
\begin{document}
\begin{abstract}
For   a prime number $p$ and any natural number $n$ we introduce, by giving an explicit recursive formula, the $p$-Jones-Wenzl projector ${}^p\JW_n$, an element of the Temperley-Lieb algebra $TL_n(2)$ with coefficients in $\FM_p$. We prove that these projectors  give
the indecomposable objects in the $\tilde{A}_1$-Hecke category over
$\FM_p$, or equivalently, they give the projector in $\mathrm{End}_{\mathrm{SL}_2(\ol{\FM_p})}((\FM_p^2)^{\otimes n})$ to the top tilting module. The way in which we find these projectors is by categorifying the fractal appearing in the
expression of the $p$-canonical basis in terms of the Kazhdan-Lusztig basis for $\tilde{A}_1$.\end{abstract}
\maketitle

\section{Introduction}

\subsection{A new paradigm} In recent years a new paradigm has emerged in  modular representation theory. The central role that   the canonical  basis of the Hecke algebra (and its associated Kazhdan-Lusztig polynomials) was believed to play  is now known to be played by the $p$-canonical basis (and its associated $p$-Kazhdan-Lusztig polynomials). The most groundbreaking papers in this direction are (in our opinion) the paper by Williamson  \cite{W} commonly known as ``Torsion explosion''  (that broke down the old paradigm), the  paper by Riche and Williamson \cite{RW} known as the ``Tilting manifesto''  (that crystallized the emerging philosophy) and the recent paper by   Achar,  Makisumi, Riche, and Williamson \cite{AMRW19} (that proved the combinatorial part of the conjecture in the tilting manifesto).

But although this brought a new scenario into place,  there was a widespread feeling that the $p$-canonical basis was impossible to calculate (if it is not by complicated categorical manipulations).  But this belief was again annihilated by the beautiful conjecture by Lusztig and Williamson  known as the ``billiards conjecture'' \cite{LW}, where they conjecture a way in which  the $p$-canonical basis (for the anti-spherical module) in type $\tilde{A}_2$ can be calculated for some finite (but big) family of elements. It is with the intention of continuing on this path that this paper comes into existence.

\subsection{The \texorpdfstring{$SL_2$}{SL2} case}
Let us consider  type $\tilde{A}_1$ (the infinite dihedral group). In this case it is easy (and known since the dawn of the theory) to obtain an explicit formula for the canonical basis.  In the paper   \cite{Eli16}, Elias lifted the canonical basis  to a categorical level in the  $\tilde{A}_1$-Hecke category over a field of characteristic zero. He obtained that the Jones-Wenzl projectors give the indecomposable objects. More precisely, there is a functor from the Temperley-Lieb category to the diagrammatic Hecke category such that the images of the Jones-Wenzl projectors give idempotents in the Bott-Samelson objects projecting to the indecomposable objects. 

The main result of this paper is an analogous result, but for fields of positive characteristic. The $p$-canonical  basis of $\tilde{A}_1$  was known since the year 2002 by the work of Erdmann and Henke \cite{EH02} (the group $SL_2$ is the only semi-simple group for which all tilting characters are known). When one expresses this basis in terms of the canonical basis one obtains a fractal-like structure (see Section \ref{cat}). We  lift  this construction to a categorical level and obtain what we call the $p$-Jones-Wenzl projectors with  recursive formulas as explicit as in the usual Jones-Wenzl projectors.

We would like to remark that the formulas for the projectors in the characteristic zero case  were not so surprising as they already appear in the Temperley-Lieb algebra. The formulas found in this paper are completely new. The most challenging and time-consuming part of the present work was to find the correct definition of the $p$-Jones-Wenzl projectors. 

\subsection{Perspectives} There are at least four possible applications of  our construction, the first one being our main motivation for this work.

\begin{enumerate}
\item Using Elias Quantum Satake \cite{ElQS} and Elias triple clasp expansion \cite{LLCC}, together with the main result of Williamson's thesis \cite{WSSB} one is not far from completely understanding the projectors giving the indecomposable objects in type $\tilde{A}_2$ over a field of characteristic zero.  The recursive formula for the Jones-Wenzl projector is built into the recursive formula for $\mathfrak{sl}_3$ (see Formulas $(1.7)$ and $(1.8)$ of \cite{LLCC}). So, as we have a $p$-analogue of this part of the formula, we would just need a $p$-analogue of the other part. If that was achieved, one would probably have the $p$-canonical basis for the whole $\tilde{A}_2$ (at least conjecturally). Of course, this might go far beyond $\tilde{A}_2$, but as the rank grows, the amount of information obtained via Quantum Satake diminishes gradually. In any case, if this approach works, it will give a good chunk of information in any rank.

\item The Jones-Wenzl projector $\JW_n$ is an endomorphism of the $n$-fold tensor product $V^{\otimes n}$ of the natural representation of the quantum group $U_q(\mathfrak{sl}_2)$ projecting onto the maximal simple module. One would like to obtain a projector satisfying the same property, but when $q$ is  a root of unity. Our $p$-Jones-Wenzl projector is certainly not the answer to this question, but might be an important ingredient.

\item In the same vein as the last point, it would be desirable to get the underlying quiver for $\mathrm{Tilt}_0(\mathrm{SL}_2)$ in prime characteristic, following the approach of \cite{RW} using the methods in \cite{AT} (the latter calculate the quiver in the root of unity case using the Jones-Wenzl projector).

\item The Jones-Wenzl projectors play a key role in the definition of Reshetikhin-Turaev $3$-manifold invariants. It is appealing to replace in that definition the Jones-Wenzl projector  by the $p$-Jones-Wenzl projector and see if one obtains a family of invariants (indexed by all primes $p$)  of some kind of object. For example, could it be that if one does this process to the colored Jones polynomial one obtains a family of invariants of framed links (necessarily more refined than the usual colored Jones polynomial)?

\item The $p$-Jones-Wenzl idempotents give the indecomposable objects in the $\tilde{A}_1$ Hecke category for the  realization obtained from the Cartan matrix with $2$ in the diagonal entries and $-2$ in the off-diagonal  entries. It would be interesting to find the idempotents corresponding to the indecomposable objects  for any realization (the $p$-Jones-Wenzl idempotents will not longer be the answer in general). In the same vein, using the ideas in \cite{EL}, one should be able to prove that stacking  $p-$Jones-Wenzl projectors side by side (as in \cite[Proposition 2.22]{EL}) one can obtain the indecomposable objects for any element of any Universal Coxeter system (again, with the realization obtained from the Cartan matrix with $2$ in the diagonal  entries and $-2$ in the off-diagonal  entries).

\end{enumerate}

\subsection{Acknowledgements} The second author presented these results at the conference ``Categorification and higher representation theory'' in July 2018 at the Mittag Leffler institute. He
would like to thank the organisers for the opportunity to discuss this construction and the participants (especially Jonathan Brundan, Daniel Tubbenhauer and Marko Stosic) for useful comments. The first author would like to thank Antonio Behn for helping him with key calculations in SAGE that  allowed us to guess the formula of the idempotents defined in this paper. We would also like to thank Geordie Williamson for interesting discussions.

The first author was supported by Beca Chile Doctorado 2017. The second author was supported by
Fondecyt 1160152 and Proyecto Anillo ACT 1415 PIA-CONICYT. The third author was supported
by Postdoctorado Fondecyt-Conicyt 3160010.

% from the work of Donkin \cite{Don93}, Doty and Henke \cite{DH05} and

%As a particular case of this view, the $p$-canonical basis of affine Weyl groups in type A control the modular representation theory (over fields of characteristic $p$) of the symmetric group or the special linear groups over ${\mathbb{F}}_p$.  The problem for a while was that

%(citar Williamson-Riche, Losev-Elias, los papers de Masuki-Achar-Williamson-Riche).

\section{Definition  of the \texorpdfstring{$p$}{p}-Jones-Wenzl idempotents}\label{pJW}

\subsection{The generic Temperley Lieb category}\label{TL}

Let  $m,n\in \NB_0 = \{0, 1, 2, \dots\}$ be such that  $n-m$ is even. An $(m,n)$-\emph{diagram} consists of the following data: 
\begin{enumerate}
\item  A closed rectangle $R$  in the plane with two opposite edges designated as top and bottom.
\item $m$  marked points (vertices) on the top edge and $n$ marked points on the bottom edge.
\item $(n+m)/2$ smooth curves (or ``strands") in $R$ such that for each curve
$\gamma$, $\partial \gamma = \gamma \cap \partial R$ consists of two of the $n+m$  marked points, and such that the curves are pairwise non-intersecting.
\end{enumerate}
Two such diagrams are {\em equivalent} if they induce the same pairing of the $n+m$ marked points. We call a $(m,n)$-\emph{crossingless matching} one such equivalence class.

Let $\delta$ be an indeterminate over $\QB$. The {\em generic Temperley Lieb category}  $\TL(\d)$ (as defined in \cite{GW}) is a strict monoidal category defined as follows. The objects are the  elements of $\NB_0.$ If $m-n$ is odd,  $\Hom(m,n)$ is the zero vector space.  If $m-n$ is even, $\Hom(m,n)$ is the $\QB(\d)$ vector space with basis  $(m,n)$-\emph{crossingless matchings}. The composition of morphisms is first defined on the level of diagrams. The composition
$g\circ f$ of an $(n,m)$-diagram $g$ and an $(m,k)$-diagram $f$  is defined by the following steps:
\begin{enumerate}
\item Put the rectangle of $g$ on top of that of $f$, identifying the top edge of $f$ (with its $m$  marked points) with the bottom edge of $g$ (with its $m$ marked points).
\item Remove from the resulting rectangle any closed loops in its interior.  The result is a $(n,k)$-diagram $h$.
\item The composition $g\circ f$ is $(-\d)^r h$, where $r$ is the number of closed loops removed.
\end{enumerate}
This composition clearly  respects equivalence of diagrams.
The \emph{tensor product of objects} in $\TL$ is given by $n\otimes n' = n+n'$. The {\em tensor product of morphisms}  is defined by horizontal juxtaposition.  With this we end the definition. 

\begin{exa}
Vertical composition in $\TC\LC(\d)$:
\begin{equation*}
\ig{1}{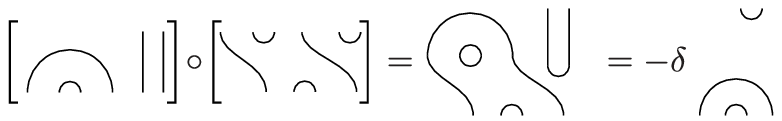}.\\
\end{equation*}
\end{exa}

Consider the \emph{flip involution}, a contravariant functor $\ol{\phantom{A}}\co \TC\LC(\d) \rightarrow \TC\LC(\d)$ defined as the identity on objects and by flipping the diagrams upside down on morphisms.

For any natural number $n$, the  \emph{Temperley-Lieb algebra on $n$ strands}  is defined to be the $\QB(\d)$-algebra $TL_n(\d):= \End_{\TC\LC(\d)}(n) $.

\begin{exa}
A generator of $TL_{12}(\d)$ as a $\QB(\delta)$-module: $\ \ \ig{1.5}{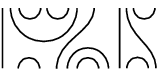}$
\end{exa}

\subsection{Jones-Wenzl projectors} \label{jw}
The results in this section are classical. For references, see the celebrated paper by Jones \cite{VJ} where the Jones-Wenzl projectors are introduced. Also see    \cite{We} for the recursion formula below and  \cite{Lick} for further properties.
Let $n$ be a natural number. Let $TL_n(2)$ be the Temperley-Lieb algebra specialised at $\d\rightsquigarrow 2$.

\begin{prop} There is a unique non-zero idempotent $\JW_n\in TL_n(2)$, called the  \emph{Jones-Wenzl projector on $n$ strands}, such that
$$e_i \circ \JW_n = \JW_n \circ e_i =0,$$ for all $1\leq i\leq n-1$, where $e_i=\ig{1}{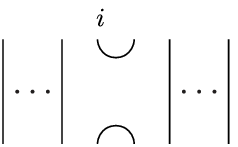}.$
\end{prop} 

It is easy to see that when the $\JW_n$ is expressed in the $\QB$-basis of $(n,n)$-diagrams,  the coefficient of the identity is $1$.

The following proposition adds-up the most important properties of the Jones-Wenzl projectors. We will prove a $p$-analogue of these properties later in the paper.
\begin{prop}\label{most} The Jones-Wenzl projectors satisfy:
\begin{enumerate}
\item \textbf{Absorption.} $\ \ \ig{1}{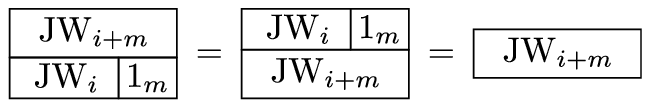}.$
\item \textbf{Recursion.} $\ \ \ig{1}{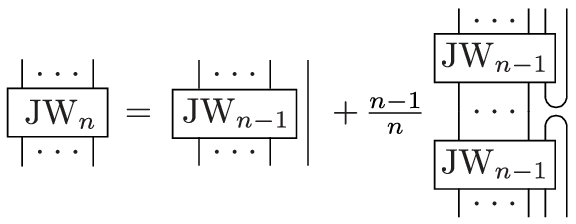}.$
\end{enumerate}
\end{prop}

As an example of the recursion,
\begin{equation*}
\ig{1}{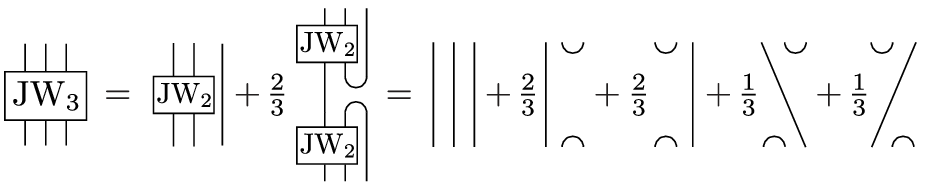}.\\
\end{equation*}

The following equality follows easily from the definitions:
\begin{equation}\label{zero}
\JW_m\circ \Hom_{\TC\LC}(n,m)\circ \JW_n= \begin{cases} \{0\} &\ \mbox{if $n\neq m$}, \\
 \spn_{\QB}\{\JW_n\}  & \mbox{ if $n=m$}.\end{cases}
\end{equation}

\subsection{Definition of the \texorpdfstring{$p$}{p}-Jones Wenzl projectors}\label{pjw}
Let us fix a prime number $p$ for the rest of this paper. We will introduce the $p$-analogue of the Jones-Wenzl idempotents defined above.

%The cardinality of $\supp_p(n)$ is $2^{k-1}$, where $k$ is the cardinality of the set $\{j\, \vert\, a_j\neq0\}$.
%Let us set a partial order $\preccurlyeq_p$ in $\NB$. We write $n\preccurlyeq_p m$ if $\supp_p(n)\subset\supp_p(m)$. A minimal number with respect to $\preccurlyeq_p$ will be called a \emph{$p$-Adam}. It has to be of the form $jp^i-1$ for some $0<j<p$ and some $i$. If $n$ is not a $p$-Adam, we call its unique predecessor the \emph{$p$-father of $n$}, and we denote it  by $f[n]$ or just $f[n]$ if $p$ is fixed beforehand.

\begin{defn} 
If $n\in \NB$ is an integer and $a_ip^i+a_{i-1}p^{i-1}+\cdots +a_1p +a_0$ is the $p$-adic expansion of $n+1$ with $a_i\neq 0$, we define the $p$-\emph{support of} $n$  to be the following set of natural numbers
\begin{equation*}
\supp_p(n)=\{a_ip^i\pm a_{i-1}p^{i-1}\pm\cdots \pm a_1p \pm a_0\}.
\end{equation*}
This set has cardinality $2^{k-1}$ where $k$ is the number of non-zero coefficients in the $p$-adic expansion of $n+1$ (in formulas, $k= \vert  \{j\in \mathbb{Z}\, \vert\, a_j\neq0\}\vert$).
 If $n+1$ has at least two non-zero coefficients in its $p$-adic expansion, we define \emph{the $p$-father of $n$} to be the natural number $f[n]$  obtained by replacing  the right-most non-zero coefficient in the $p$-adic expansion of $n+1$ by zero and then substracting $1$. In formulas, if $n+1=\sum_{i=m}^r a_ip^i$ with $a_m\neq 0$, then $f[n]\coloneqq (\sum_{i=m+1}^r a_ip^i)-1$. If we want to remark the dependence on $p$ we will denote it  by $f_p[n]$ (this will only be done in examples \ref{ex1} and \ref{ex2} at the end of this section). 
If $n+1$ has only one non-zero coefficient in its $p$-adic expansion, then $n+1=jp^i$ for some $0<j<p$ and some $i\in\NM$. In that case, we say that $n$ is a \emph{$p$-Adam} (because it has no father).\end{defn}

\begin{nota}
If $J\subseteq \mathbb{N}$ and $n\in \mathbb{N}$, we will denote $J+n\coloneqq \{j+n\, \vert\, j\in J\}$. We will also denote, when $p$ is clear from the context,  $I_n\coloneqq \supp_p(n)-1$. 
\end{nota}

We fix a prime number $p$. Let us define the  \emph{rational $p$-Jones-Wenzl idempotent on $n$ strands}, denoted by ${}^p\JW_n^{\QB}$.  We will write it down in the following form
\begin{equation}\label{U}
\ig{1}{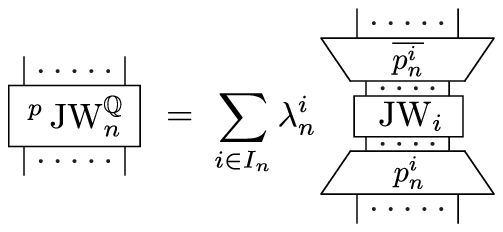},
\end{equation}
%\begin{equation}\label{U}
%{}^p\JW_n^{\QB}:=\sum_{i\in I_n}\l_n^i (\overline{p_n^i} \circ \JW_i \circ p_n^i).
%\end{equation}
with  $\l_n^i \in \QB$ and  $p_{n}^i\in \mathrm{Hom}(n,i)$.  We will define $ \l_n^i$ and $  p_n^i $  inductively on the number of non-zero coefficients in the $p$-adic expansion of $n+1$.  

 If $n$ is a $p$-Adam, we define $$\ \ig{1}{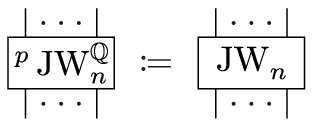},$$
or to be more precise, as  $\mathrm{supp}_p(n)-1=\{n\},$ we define $\l_n^n=1$ and $p_n^n =\mathrm{id}\in TL_n(2).$

 If $n$ is not a $p$-Adam, %, suppose that 
  let us denote by  $m \coloneqq n-f[n]$.  By induction hypothesis we suppose that $\l_{f[n]}^i$ and $p_{f[n]}^i$ are known. We define
\begin{equation}\label{pJWn}
\ig{1}{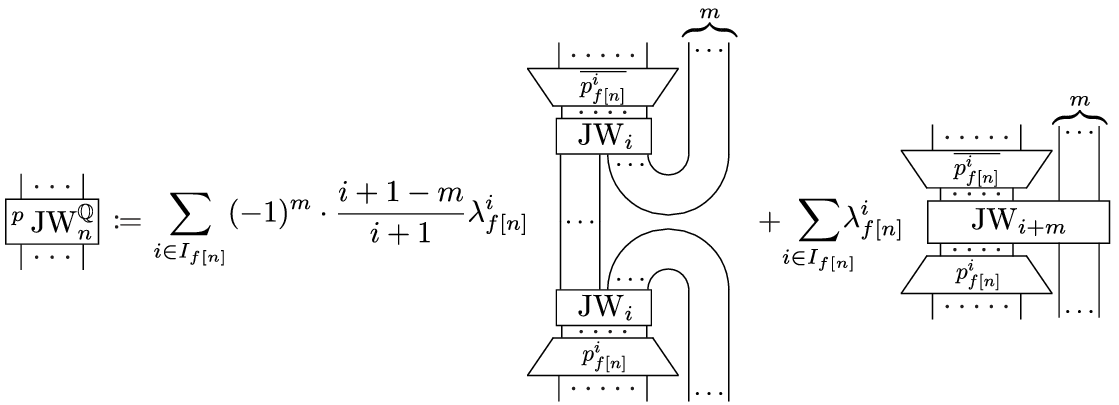},
\end{equation}
Let us be more precise. We have that $$I_n=(I_{f[n]}-m)\sqcup (I_{f[n]}+m).$$ 
We remark that the union is disjoint because  if $a_ip^i+a_{i-1}p^{i-1}+\cdots +a_rp^r $ is the $p$-adic expansion of $n+1$, then $a_rp^r=m$ and
$$\supp_p(n)=\{a_ip^i\pm a_{i-1}p^{i-1}\pm\cdots  - a_rp^r\}  \sqcup \{a_ip^i\pm a_{i-1}p^{i-1}\pm\cdots  + a_rp^r\}.$$

If $i\in  I_{f[n]}$ we have that 
 $$ \l^{i-m}_n=(-1)^m\cdot \frac{i+1-m}{i+1}\l^i_{f[n]}\  , \ \ \l_n^{i+m}=\l^i_{f[n]} \ , $$ 

$$\ig{1}{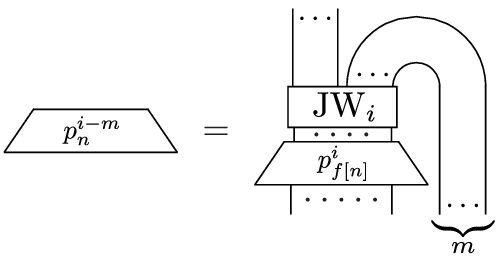} \ \  \   \mathrm{and} \  \ \ \ig{1}{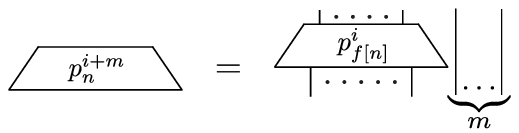}.$$
With this we finish the definition of ${}^p\JW_n^{\QB}$.

\begin{nota} \label{NU} Under the conventions above we will denote
$U_{n}^i:=\overline{p_n^i} \circ \JW_i \circ p_n^i$ (where $\overline{p_n^i}$ is the image of $p_n^i$ under the flip involution).
\end{nota}

\begin{thm}\label{main}
For all $n\in \NB$, the morphism ${}^p\JW_n^{\QB}\in TL_n(2)$ is an idempotent. Furthermore, if we express ${}^p\JW_n^{\QB}$ in the $\QB$-basis of crossingless matchings, and write each of its coefficients as an irreducible fraction $a/b$, then $p$ does not divide $b$.
\end{thm}

We remark that in the definition of the Temperley-Lieb algebra one could have used any other commutative ring $\RB$ instead of $\QB$. We denote by $TL_n(2)_{\RB}$ the corresponding algebra.
Now we can state the main definition of this paper. 
\begin{defn}
We define the \emph{$p$-Jones-Wenzl projector on $n$-strands} ${}^p\JW_n\in TL_n(2)_{\FM_p}$  as the expansion of
 ${}^p\JW_n^{\QB}\in TL_n(2)$ in the $\QB$-basis of crossingless matchings  but replacing each of the coefficients $a/b$ (expressed as irreducible fractions) by  ${\ol{a}\cdot}\left(\ol{b}\right)^{-1}\in\FM_p$, where the bar means  reduction modulo $p$.
\end{defn} 

\begin{rem} 
A  more elegant way to define the $p$-Jones Wenzl projector is to lift ${}^p\JW_n^{\QB}\in TL_n(2)_{ \QB}\subset TL_n(2)_{\QB_p}$ to an idempotent in $  TL_n(2)_{\ZB_p}$ and then project to $  TL_n(2)_{\FM_p}$.
\end{rem}

\begin{rem}
One can define an analogue of ${}^p\JW_n^{\QB}\in TL_n(2)$  in the generic Temperley-Lieb algebra $TL_n(\d)$. This is done  by replacing natural numbers by quantum numbers in the coefficients of the formula. For instance,  $(i+1-m)/(i+1)$ must be changed  by $[i+1-m]_q/[i+1]_q\in\QB(\d)$. The projectors thus defined satisfy all properties in Section \ref{idem} with  essentially  the same proofs. 
\end{rem}

\begin{exa}[Example of a rational $3$-Jones-Wenzl projector]\label{ex1}
Let us compute ${}^3\JW_{10}^{\QB}$. We notice that  $f_3[10]=8$ and that $8$ is a $3$-Adam. Using \eqref{pJWn} we have,
\begin{equation}
\ig{1}{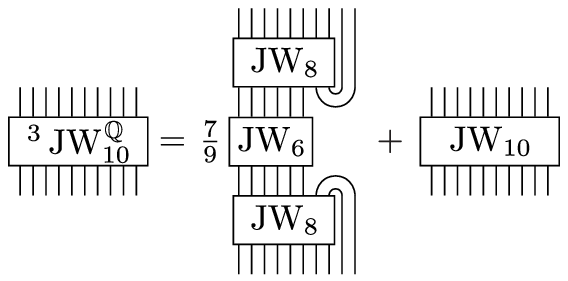}.
\end{equation}
\end{exa}

\begin{exa}[Example of rational $2$-Jones-Wenzl]\label{ex2}
To calculate ${}^2\JW_{10}^{\QB}$, first we note that $f_2[10]=9$, $f_2[9]=7$ and $7$ is a $2$-Adam. Using \eqref{pJWn} we have,
\begin{equation}
\ig{1}{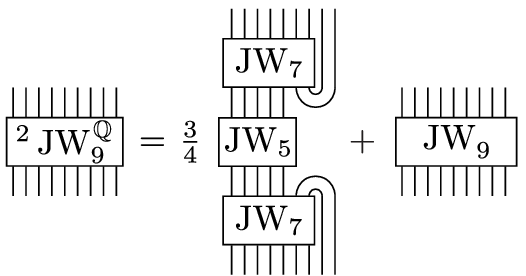}.
\end{equation}
Using \eqref{pJWn} again we obtain,
\begin{equation}
\ig{1}{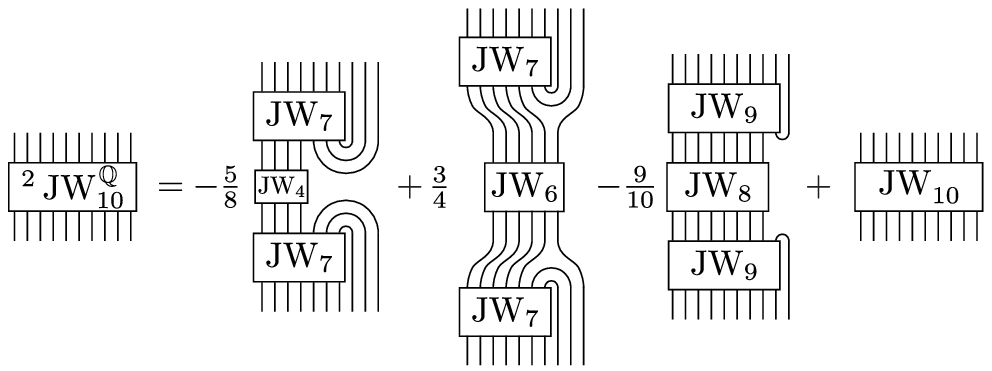}.
\end{equation}
Note that  ${}^3\JW_{10}^{\QB}$ and  ${}^2\JW_{10}^{\QB}$ are quite different.
\end{exa}

%Unfortunately, in the previous two examples, it was impossible for us to give the $p$-Jones-Wenzl with coefficients in $\FB_p$. The only way we know to do it is by expanding the idempotents in the Temperley-Lieb standard basis. However, this kind of computations (for numbers like $10$ or bigger) are extremely slow, even by using computers. That is the reason that our recursion ins extremely useful, it is quite straightforward to express them in terms of previous ones.

\section{Some properties of the $p$-Jones-Wenzl projectors}\label{demo}
\subsection{}\label{idem}

The following lemma, although simple, will prove to be useful. 
\begin{lem} \label{lema lambda}
  Let $0 \leqslant m \leqslant n$. In $TL_n(2)$ we have the equality
   \begin{equation*}
\ig{1}{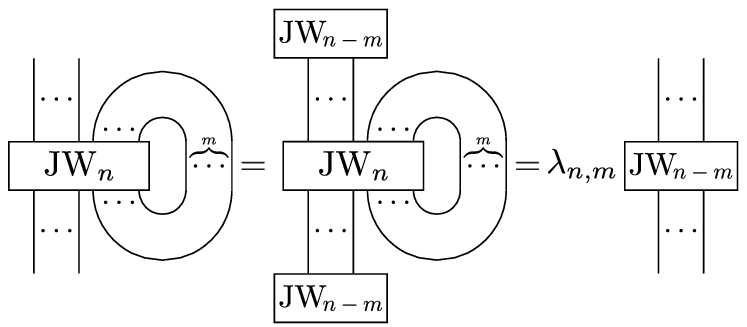},
\end{equation*}
where $\l_{n,m}:=(-1)^{m} \cdot \frac{n+1}{n+1-m}$.
\end{lem}
\begin{proof}
  The first equality is a consequence of Proposition \ref{most}; moreover, from \eqref{zero} we can deduce the existence of some coefficients $\l_{n,m}\in \QB$ satisfying the second equality. We only need to calculate these $\l_{n,m}$ to finish the proof. 
  Let us observe that $\l_{n,0}=1$ and $\l_{n,m}=\l_{n,k}\cdot \l_{n-k,m-k}$, for all $0 \leqslant k \leqslant m$. We prove the result by induction on $m$. For $m=1$ we have that $\l_{n,1}=-(n+1)/n$, by \cite[Eq. $(2.8)$]{EL}. Let $m>1$. By our inductive hypothesis we obtain
\begin{equation*}
\l_{n,m}=\l_{n,1}\cdot \l_{n-1,m-1}=-\frac{(n+1)}{n} \cdot \frac{(-1)^{m-1}\cdot n}{n-(m-1)}=(-1)^m \cdot \frac{n+1}{n+1-m}.
\end{equation*}\end{proof}

%\subsection{Idempotence and orthogonal decomposition}\label{idem}
%We will prove that ${}^p\JW_n^{\QB}$ as defined in \ref{pjw} is an idempotent.

%Let $TL_n(2)$ be the Temperley-Lieb algebra over $\QB$. By \ref{zero} we have that
%\begin{equation*}
%\ig{1}{lmn},
%\end{equation*}
%for some $\l_{n,m}\in\QB$, the first equality is due to first part of \ref{most}.
%It is known that $\l_{n,1}=-(n+1)/n$, see \cite[eq. 2.8]{EL}. Furthermore, it is clear that the $\l$'s satisfy the multiplicative formula $\l_{n,k}\cdot \l_{n-k,m-k}=\l_{n,m}$. We will prove
%\begin{equation*}
%\l_{n,m}=(-1)^{m} \cdot \frac{n+1}{n+1-m}
%\end{equation*}
%by induction on $m$,
%\begin{equation*}
%\l_{n,m}=\l_{n,1}\cdot \l_{n-1,m-1}=-\frac{(n+1)}{n} \cdot \frac{(-1)^{m-1}\cdot n}{(n-(m-1))}=(-1)^m \cdot \frac{n+1}{n+1-m} .
%\end{equation*}

\begin{prop}\label{sidem}
The element $^p\JW_n^{\QB}\in TL_n(2)$ is an idempotent. Moreover,  $\{\lambda_n^i U_n^i\}_{i\in I_n}$ is a set of mutually orthogonal idempotents. 
\end{prop}
\begin{proof}
We will prove it by induction in the number of non-zero terms that $n+1$ has in the $p$-adic expansion. If $n$ is a $p$-Adam, then $^p\JW_n^{\QB}=\JW_n$, which is an idempotent. Consider now $n$ not to be a $p$-Adam. We have that $$^p\JW_{f[n]}^{\QB}=\sum\limits_{i\in \supp_p(f[n])-1} \lambda_{f[n]}^i \left(\overline{p_{f[n]}^i } \circ \JW_i \circ p_{f[n]}^i \right).$$  By our inductive hypothesis and  Equation \eqref{zero}, we have that  
\begin{equation}\label{formulauno}
\JW_i \circ p_{f[n]}^i  \circ \overline{p_{f[n]}^i } \circ \JW_i =\frac{1}{\lambda_{f[n]}^i } \JW_i
\end{equation} and
$\JW_i \circ p_{f[n]}^i  \circ \overline{p_{f[n]}^j} \circ \JW_j =0$, for all $i\neq j\in \supp_p(f[n])-1$. Then, absorption and \eqref{formulauno} give
\begin{equation*}
  \ig{1}{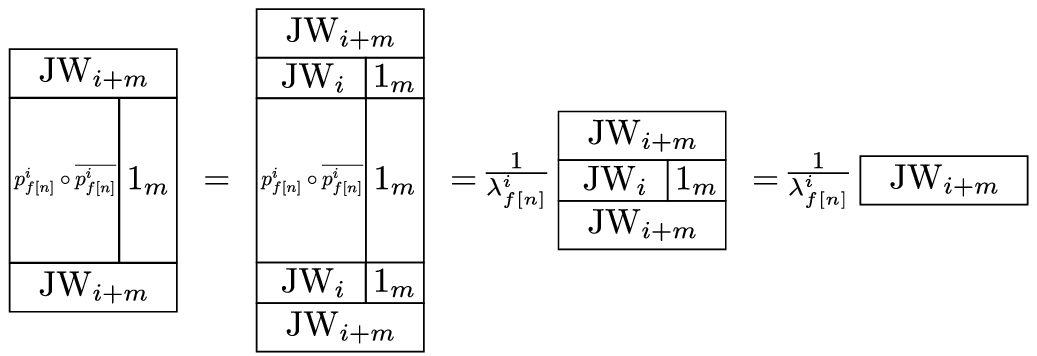},
\end{equation*} 
Equation \eqref{formulauno} and Lemma \ref{lema lambda} give
\begin{equation*}
 \ig{1}{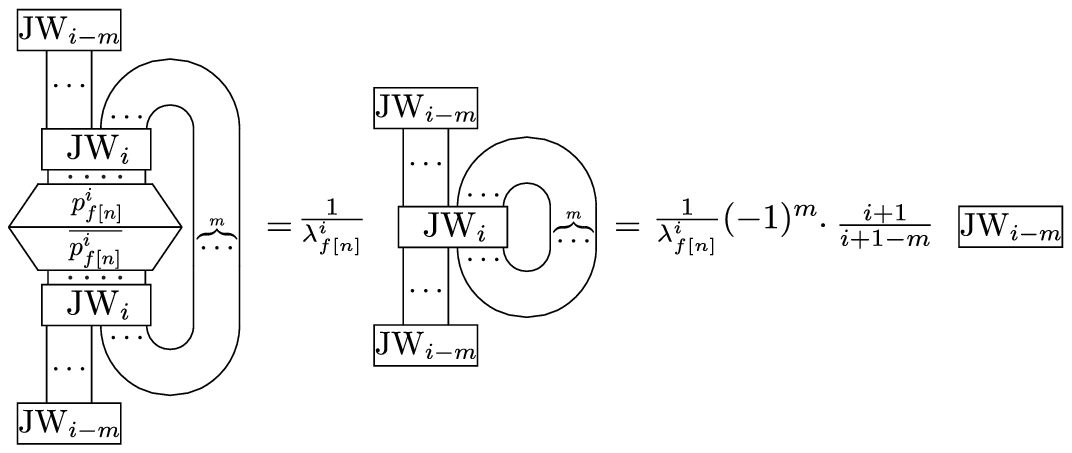}.
\end{equation*} By \eqref{pJWn}, these two formulas prove the idempotence of each summand in $^p\JW_n^{\QB}$. Since $i\pm m \neq j \pm m$ for all $i,j\in \supp_p(f[n])-1$, $i\neq j$, by \eqref{zero} we finish the proof. 
\end{proof}

%\begin{thm}\label{sidem}
%Let $n$ a non-$p$-Adam number. If $\ig{1}{hip}$ then
%\begin{equation*}
%\ig{1}{hip0},
%\end{equation*}
%for all $i\in I\subset\NB$.
%\end{thm}
%\begin{proof}
%Let us prove it by induction with respect to $\preccurlyeq$. Let $n$ a non-$p$-Adam number. By induction hypothesis suppose that $\ig{1}{hip}$, and also $\ig{1}{hip0}.$
%Then we have to prove the theorem for $n$. By the first part of \ref{most} we have
%\begin{align*}
%\ig{1}{idem1},\\
%\ig{1}{idem2}.
%\end{align*}
%\end{proof}

%In particular, each summand of ${}^p\JW_n^{\QB}$ is an idempotent. By construction of rational $p$-Jones-Wenzl, it is easy to note that for all $i,j\in I$ we have $i\pm m\neq j\pm m$. Using \ref{zero} it follows that those summands are orthogonal each other. Then ${}^p\JW_n^{\QB}$ is an idempotent as a sum of orthogonal idempotents.

%\subsection{Absorption}\label{absor}

%In this section we will prove an analogue result to the first part of \ref{most} with the rational $p$-Jones-Wenzl idempotents. For a shorthand let us write $\sum\l_i U_{f[n]}^i$ to the orthogonal decomposition of ${}^p\JW_{f[n]}^{\QB}$, where $U_n^i\in TL_n$ and $\l_i\in\QB$. Let $m=n-f[n]$ and let us write $\sum\mu_i U_{n}^{i+m}+\sum\eta_i U_{n}^{i-m}$ to the orthogonal decomposition of ${}^p\JW_n^{\QB}$, where $\mu_i,\eta_i\in\QB$.

\begin{prop}\label{abs}  The idempotent $^p\JW_n^{\QB}$ satisfies the following absorption property:
\begin{equation*}
\ig{1}{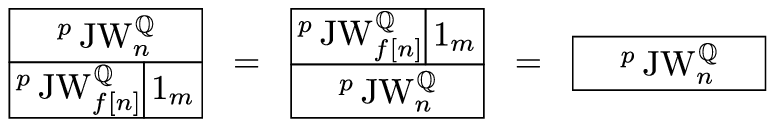}.
\end{equation*}
\end{prop}
\begin{proof}
We prove only the second equality, the first one being analogous. Recall Equation \eqref{U} and Notation \eqref{NU} and remark that
\begin{equation*}
\ig{1}{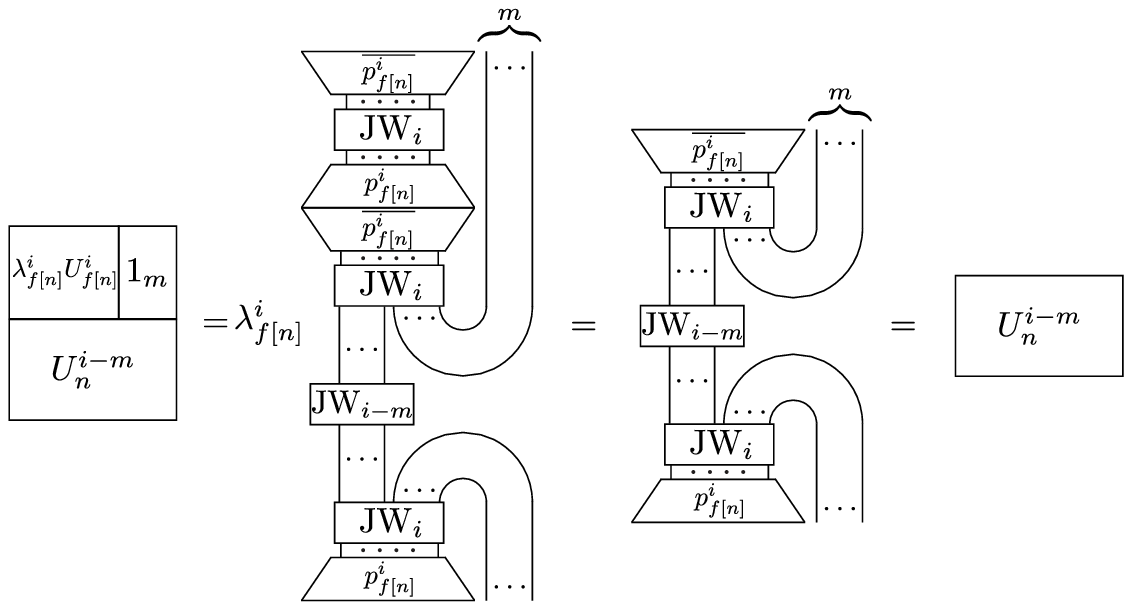}.
\end{equation*}
By the first part of Proposition \ref{most} and Equation \eqref{formulauno} we have that
\begin{equation*}
\ig{1}{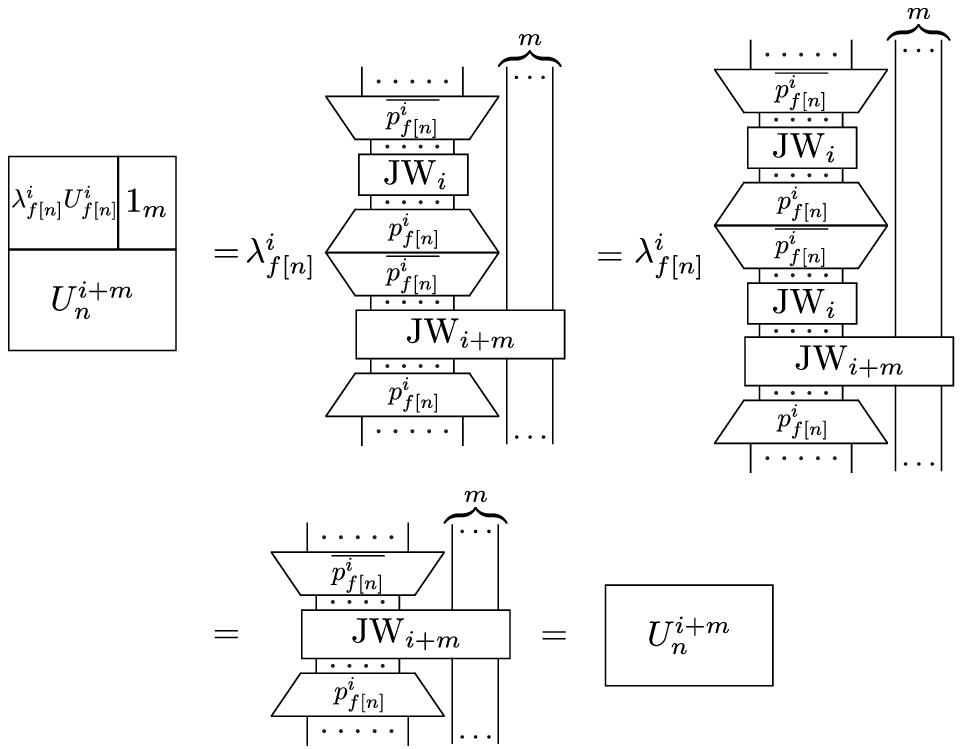}
\end{equation*}
By \eqref{pJWn} these two formulas prove the proposition. This is because the remaining terms appearing in the expansion of the left-hand side are all zero by Equation \eqref{zero}.

\end{proof}

%\begin{thm}\label{abs} The rational $p$-Jones-Wenzl idempotents satisfy the following formula
%\begin{equation*}
%\ig{1}{abs}.
%\end{equation*}
%\end{thm}
%\begin{proof}
%We will prove only the second equality, the first is analogous. Note that in the LHS appear the term,
%\begin{equation*}
%\ig{1}{ab1}.
%\end{equation*}
%And by the first part of \ref{most}, we have on the LHS.
%\begin{equation*}
%\ig{1}{ab2}
%\end{equation*}
%The remaining terms of the LHS are of the form
%\begin{align*}
%U_n^{i-m}\circ \lambda_j  \left(U_{f[n]}^j\otimes 1_m\right)&=U_n^{i-m}\circ \left(\left(\lambda_i  U_{f[n]}^i\circ \lambda_j  U_{f[n]}^j\right)\otimes 1_m\right)=0,\\
%U_n^{i+m}\circ \lambda_j  \left(U_{f[n]}^j\otimes 1_m\right)&=U_n^{i+m}\circ \left(\left(\lambda_i  U_{f[n]}^i\circ \lambda_j  U_{f[n]}^j\right)\otimes 1_m\right)=0.
%\end{align*}
%Where $i\neq j$. They are $0$ by the orthogonality of the $ U_{f[n]}^i$'s. This finishes the proof.
%\end{proof}

\section{The Hecke category of \texorpdfstring{$\tilde{A_1}$}{A1}-Soergel bimodules}\label{hecke}

\subsection{Hecke algebra} The \emph{infinite dihedral  group} $U_2$ (of type $\tilde{A_1}$) is the group with  presentation $U_2=\langle s,t:s^2=t^2=e\rangle$. We denote the length function by $\ell$ and the Bruhat order by $\leq$. %The elements of the generating set $S\coloneqq \{s,t\}$ are called the \emph{simple reflections} of $\tilde{A_1}$.We will use two colors to represent the elements of $S$, e.g., $S=\{{\sa},{\tr}\}$. 
An \emph{expression} is an ordered tuple $\ul{w}=(s_1,s_2,\ldots, s_r)$ of elements of $S$.
%; it has length $r$ and it is reduced if $s_i\neq s_{i+1}$ for all $i$. 
We denote by $w\in W$  the corresponding product of simple reflections $w=s_1s_2\cdots s_r$. If $l(w)=r,$ we say that the expression is \emph{reduced}.

Consider the ring $\LC=\Zvv$ of Laurent polynomials with integer coefficients in one variable $v$. 
The \emph{Hecke algebra} $\HB$ of  the infinite dihedral group is the free $\LC$-module with basis $\{H_w\, \vert \, w\in U_2\}$ and multiplication given by:
$$H_wH_s=\left\{
     \begin{array}{ll}
       H_{ws} & \hbox{if $w<ws$;} \\
       H_{ws}+(v^{-1}-v)H_w & \hbox{if $ws<w$,}
     \end{array}
   \right.
 $$ for all $w\in U_2$.
 The set $\{H_w: w\in U_2\}$ is called the \emph{standard basis} of $\HB$. On the other hand, $\HB$ has the \emph{Kazhdan-Lusztig basis} (or \emph{KL-basis}) that we call $\{b_w: w\in U_2\}$. In the literature this basis is also denoted by $\underline{H}_w$ (see \cite{So1}) or $C'_w$ in the original paper by Kazhdan and Lusztig \cite{KL79}. The following formula has an easy proof (all the calculations in the infinite dihedral group are explicit).

\begin{lem}\label{dyerr} Let $(s_1, s_2,\ldots, s_k)$ be a reduced expression of $w\in U_2$ and $r$ a simple reflection. Then
  $$b_wb_r=\left\{
             \begin{array}{ll}
               (v+v^{-1})b_w & \hbox{if $r=s_k$;} \\
               b_{wr}+b_{ws_k} & \hbox{if $k>1$ and $r= s_{k-1}$;}\\
               b_{wr} &\hbox{otherwise.}
             \end{array}
           \right.
  $$
   \end{lem}

\subsection{The \texorpdfstring{$p$}{p}-canonical basis}\label{pcan}
 Consider the Coxeter system $W=U_2$. Let $\HC$ be the Hecke category  (as defined in \cite{EW16})  over $\ZM$ with a minimal realization obtained from the  Cartan matrix 
 \begin{equation*}
 \left(
\begin{array}{cc}
2 & -2 \\
-2 & 2 
\end{array} \right).
\end{equation*}
This defines a category $\HC^{\Bbbk}$ by base change, for any ring $\Bbbk.$ The following theorem is a summary of the properties of $\HC^{\Bbbk}$ when $\Bbbk$ is a complete local ring \cite[Lemma 6.25, Theorem 6.26 and Corollary 6.27]{EW16}.
\begin{thm}
Let $\Bbbk$ be a complete local ring. 
\begin{itemize}
\item The category $\HC^{\Bbbk}$ is a Krull-Remak-Schmidt $\Bbbk$-linear category with a grading shift functor $(1)$.
\item The indecomposable objects $B_w$ are indexed by $w\in W$ (modulo grading shift) and $B_w\sumset \ul{w}$ is the unique summand of $\ul{w}$ (where $\ul{w}$ is any reduced expression of $w\in W$) that does not appear in any reduced expression $\ul{u}$ with $u\leq w$.  
\item If  $\left<\HC^{\Bbbk}\right>$ denotes the split Grothendieck group of $\HC^{\Bbbk}$, then $\left<\HC^{\Bbbk}\right>$ has a $\Zvv$-algebra structure as follows: the monoidal structure on $\HC^{\Bbbk}$ induces a unital, associative multiplication and $v$ acts via $v\left<B\right> \coloneqq \left<B(1)\right>$ for an object $B$ of $\HC^{\Bbbk}$. Then, there is an isomorphism of $\Zvv$-algebras called the \emph{character}
\begin{equation*}
\mathrm{ch} \co \left<\HC^{\Bbbk}\right> \longrightarrow \HB
\end{equation*}
with $\mathrm{ch}(\left<B_s\right>)=b_s$ for all $s \in S$.
\end{itemize} 
\end{thm}

Another  fundamental theorem by Elias and Williamson \cite{EW14} (conjectured by Soergel)  is the following. 
\begin{thm}In $\HC^{\RB},$ the image of the indecomposable objects are the Kazhdan-Lusztig basis. In formulas: 
 $$\mathrm{ch}( \left<B_w\right>)= b_w.$$
\end{thm}
 This result again was proved in much greater generality (for any Coxeter group and for any realization satisfying certain positivity conditions). 
However, in $\HC^{\F_p}$ this is not the case. In this latter category, to emphasize the dependence on $p$, we will denote by  ${}^p B_w$ the indecomposable object. Let us define $$\mathrm{ch}(\left<{}^p B_w\right>):={}^p b_w.$$
The set $\{{}^p b_w\}_{w\in W}$ is another $\Zvv$-basis of $\HB$. It is called the \emph{$p$-canonical basis of $\HB$} (see \cite{JW}).

\subsection{The Jones-Wenzl idempotents as Soergel bimodules}\label{functor}
For the infinite dihedral group, let us color the  set $S=\{{\sa},{\tr}\}$. Let $\TC\LC_c$ be the $2$-colored Temperley-Lieb category as defined in \cite[Section 2.1]{EL} (this is just the generic Temperley-Lieb category defined in Section \ref{TL}, specialised at $\delta\rightsquigarrow 2$,  with regions colored by elements of $S$ in such a way that  adjacent regions always have  different colours). Take a diagram $\EC$ in $\TC\LC(2)$ and colour its regions accordingly using the set $S$  with ${\sa}$ colouring its left-most region. This new diagram ${}^{\sa}\EC$ is a morphism in $\TC\LC_c$. By abuse of notation we will just call it $\EC$. For example,
\begin{equation*}
\ig{1.5}{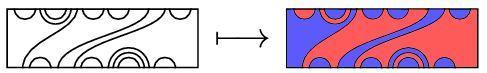}.
\end{equation*}

%Consider the Jones-Wenzl idempotent $\JW_n$ defined in Section \ref{jw}, and apply to it this colouring to obtain an idempotent $JW_n\in\TC\LC_c$. 
The following result was proven in \cite[Theorem 5.29]{Eli16}.
\begin{thm}
There is an additive $\QB$-linear (non-monoidal) faithful functor \linebreak  $\FS\colon \on{Kar}(\TC\LC_c)\longrightarrow \HC^{\QB}$, where $\on{Kar}(-)$ is the Karoubi envelope functor. The functor $\FS$ is fully faithful if one only considers degree zero morphisms in $\HC^{\QB}$ and takes the object $(n,\JW_n)$ into the indecomposable object $B_{\ul{n+1}}$, where ${\ul{n}}$ is the unique element $w$ of length $n$ such that ${\sa}w<w$ in $U_2$. 
\end{thm}
The definition of the functor $\FS$ is given in  \cite[Definition 5.14]{Eli16} (see example 5.15 in that paper to get a quick understanding of that definition).  
\subsection{Categorification}\label{cat}
\begin{defn}
 In  \cite{EH02} the authors define the \emph{admissible expansion} of $n\in \mathbb{Z}$ as the unique expansion $n=\sum\limits_{i=0}^l n_ip^i$ with $p-1\leqslant n_i \leqslant 2p-2$ for $i<l$ and $0\leqslant n_l< p-1$. (The uniqueness is proved in \cite[Lemma 5]{EH02}.)
 \end{defn}
 \begin{rem} In \cite{EH02} and in \cite{JW} there is a minor mistake in the definition: they consider $0\leqslant n_l\leqslant p-1$. With that definition the expansion is non-unique: if $n_l=p-1$ one could define $n_{l+1}=0$.
\end{rem}

\begin{lem}\label{tec}
Let $ n=\sum\limits_{i=0}^l n_ip^i$  be written in its admissible expansion. Let $ m=\sum\limits_{i=0}^l m_ip^i$ (not assumed to be written in its admissible expansion). Then $m_i\in \{n_i,2p-2-n_i\}$ for all $i<l$ and $m_l=n_l$ if and only if $m+1\in \supp_p(n)$.
\end{lem}
\begin{proof}

Let $ a=\sum\limits_{i=0}^l a_ip^i$  be the $p$-adic expansion of the number $a\in \mathbb{N}$. We define $[i]^p(a):=a_i$ for all $i \in \mathbb{N}$.

Consider the following expansion $\mathbb{N}\ni b=\sum\limits_{i=0}^l b_ip^i$ (not assumed to satisfy any property).  We have that 
\begin{equation}\label{expansion}
b+1=[b_0-(p-1)]+\cdots +[b_{l-1}-(p-1)]p^{l-1}+(b_l+1)p^l.
\end{equation}
This equation applied to the $p$-adic expansion of $ n$  gives us that $[i]^p(n+1)=n_i-(p-1)$ for all $i<l$ and $[l]^p[n+1]=n_l+1$.

We have the following easy facts:

\begin{itemize}
  \item $m_i=n_i$ if and only if $m_i-(p-1)=[i]^p(n+1)$ for all $i<l$.
  \item $m_i=2p-2-n_i$ if and only if  $m_i-(p-1)=-[i]^p(n+1)$ for all $i<l$.
  \item $m_l=n_l$ if and only if $m_l+1=n_l+1$.
\end{itemize}
These facts and equation \eqref{expansion} applied to the expansion $ m=\sum\limits_{i=0}^l m_ip^i$ give us
$$m+1\in \{\pm[0]^p(n+1)p^0\pm[1]^p(n+1)p^1 \pm \cdots \pm [l-1]^p(n+1)p^{l-1}+(n_l+1)p^l\},$$
thus finishing the proof. 
\end{proof}

\begin{prop}\label{deca1}The $p$-canonical basis can be expressed in the following way:
$${}^p b_{\ul{n+1}}=\sum_{i\in \supp_p(n)}  b_{\ul{i}}.$$
\end{prop}
\begin{proof}
By \cite[Lemma 6]{EH02} and \cite[Section 5.3]{JW} (where the explicit relation to the $p$-canonical basis is given) this proposition is a rephrasing of Lemma \ref{tec}.
\end{proof}

We will categorify the formula given in Lemma \ref{deca1}. Recall that  $\sum_{i\in \supp_p(n)-1}\l_n^i U_{n}^i$ is the orthogonal decomposition of ${}^p\JW_n^{\QB}$, as in \eqref{U}, where  $U_{n}^i:=\overline{p_n^i} \circ \JW_i \circ p_n^i.$ 

\begin{prop}
In the category $\on{Kar}(\TC\LC_c)$ there is an isomorphism of objects
\begin{equation*}
(n, {}^p\JW_n^{\QB})\cong \bigoplus_{i\in \supp_p(n)-1}(i, \JW_i).
\end{equation*}
\end{prop}
\begin{proof}
Since  $\{\l_n^i U_{n}^i\}_{i\in \supp_p(n)-1}$ is  a set of mutually orthogonal projectors, it is enough to prove that $(n,\l_n^i U_{n}^i)\cong (i,\JW_i)$. Consider the map $$f=\l_n^i(\JW_i\circ  p_n^i): (n,\l_n^i U_{n}^i)\rightarrow (i,\JW_i). $$ By Equation \eqref{formulauno} one can see that $f$ is indeed a map in the Karoubi envelope, i.e., $f=\JW_i\circ f\circ (\l_n^i U_{n}^i).$  It is easy to prove that $g= \ol{p_n^i}\circ \JW_i$ is the inverse of $f$, thus proving the proposition.
\end{proof}
Applying the functor $\FS$ one obtains $$\FS\left((n,{}^p\JW_n^{\QB})\right)\cong \bigoplus_{i\in \supp_p(n)}B_{\ul{i}}.$$ Finally, we decategorify and apply the character $\mathrm{ch}$ defined in Section \ref{pcan} to obtain 
\begin{equation}\label{deca2}
\mathrm{ch} \left(\left\langle \FS\left((n,{}^p\JW_n^{\QB})\right)\right\rangle \right)=\ {}^p b_{\ul{n+1}}.
\end{equation}

\subsection{The absorption property determines the rational $p$-Jones Wenzl projector}
%Let us denote by ${}^p\SBim$ to the diagrammatic category of $p$-Soergel bimodules, it is equal to the category $\on{Kar(\SD)}\otimes \FB_p$. Our last task is to prove our main theorem \ref{main}. Let $w=\ul{n}$,  $v=\ul{f[n]}$ and $m=n-f[n]$ for the rest of this section. The $p$-Soergel bimodule ${}^p B_w$ tautologically decategorifies into ${}^p b_{w}$. Our goal is to lift in a nice way ${}^p\JW_n^{\QB}$ to ${}^p B_w$ and call it ${}^p\JW_n$.

%\begin{nota}For $k\in\NB$, let us write\begin{align*}
%b_{\sa}^k\coloneqq b_{\sa}\cdot b_{\tr}\cdots\\
%b_{\tr}^k\coloneqq b_{\tr}\cdot b_{\sa}\cdots
%\end{align*}
%where in the right-hand side there are $k$ alternating terms. 
%\end{nota}

\begin{nota} Consider $m\in \NB$.
If $n$ is even, we denote $ ^pb_{\ul{n}}b_1^m\coloneqq\  ^pb_{\ul{n}}\underbrace{b_sb_tb_s\cdots}_{m\ \mathrm{terms}} $. If $n$ is odd, we denote $ ^pb_{\ul{n}}b_1^m\coloneqq\  ^pb_{\ul{n}}\underbrace{b_tb_sb_t\cdots}_{m\ \mathrm{terms}} $.
\end{nota}

%If $n\in \mathbb{N}$ we denote   $m:=n-f[n]$. If $\supp_p{f{[n]}}={m_j}$, then $\supp_p{n}={m_j}\pm m$. We have the following lemma.

\begin{lem}\label{lemma} If $n\in \NB$ and $m:=n-f[n]$, there is a finite set $K\subset \NB$ such that 
\begin{equation*}
^pb_{\ul{f[n]+1}}b_{1}^{m}= \ ^pb_{\ul{n+1}}+\sum_{k\in K}c_k\,   b_{\ul{k}}
\end{equation*}
where $k\notin \supp_p(n)$ and $c_k\in \NB$ for all $k\in K$.
\end{lem}

\begin{proof}
By Lemma \ref{deca1} we have that $$ ^pb_{\ul{f[n]+1}}=\sum_{j\in J} b_{\ul{j}}$$ with $J:=\mathrm{supp}(f[n]).$

By using $m$ times  Lemma \ref{dyerr} we have,
\begin{align*}
^pb_{\ul{f[n]+1}}b_{1}^{m}&=\sum_{j\in J}b_{\ul{j}}b_{1}^{m}\\
       &=\sum_{j\in J}\sum_{r=0}^{m} {m\choose r} b_{\ul{j-m+2r}}\\
       &=\sum_{j\in \mathrm{supp}(n)} b_{\ul{j}} + \sum_{j\in J}\sum_{r=1}^{m-1} {m\choose r} b_{\ul{j-m+2r}}\\
       &={}^pb_{\ul{n+1}}+ \sum_{j\in J}\sum_{r=1}^{m-1} {m\choose r} b_{\ul{j-m+2r}}.
\end{align*}
since $0<r<m$, this concludes the lemma.
\end{proof}

The following corollary gives an alternative definition of the rational $p$-Jones-Wenzl projector. 

\begin{cor}
The projector  $ \FS\left({}^p\JW_{n}^{\QB}\right)$ is the unique idempotent in 
 the endomorphism ring of the object $  \FS\left( (n, {}^p\JW_{f[n]}^{\QB}\otimes \mathrm{id}_m )\right)  \in \HC^{\mathbb{Q}}$  whose image is isomorphic to that of $ \FS\left({}^p\JW_{n}^{\QB}\right)$ (or, in other words, whose image categorifies $ {}^p b_{\ul{n+1}}$). 
\end{cor}

\begin{proof}
We recall  that in $\HC^{\mathbb{Q}}$,  the degree zero part of $\mathrm{Hom}(B_x,B_y)$ is either $\mathbb{Q}\cdot \mathrm{id}$ if $x=y $ or zero if $x\neq y.$ The absorption property (Proposition \ref{abs}) means that $ \FS\left((n,{}^p\JW_{n}^{\QB})\right)$ is a direct summand of $ \FS\left( (n, {}^p\JW_{f[n]}^{\QB}\otimes \mathrm{id}_m )\right)$,
so by  Lemma \ref{lemma} the result follows. 

% $ {}^p\JW_{n}^{\QB}$ is an element in  the endomorphism ring of of the object $$   (n, {}^p\JW_{f[n]}^{\QB}\otimes \mathrm{id}_m )  \in \on{Kar}(\TC\LC_c),$$ 
 
%$ \FS\left((n,{}^p\JW_{n}^{\QB})\right)$ is a direct summand of $  \FS\left( %{}^p\JW_{f[n]}^{\QB}\right)\otimes \mathrm{id}^m  $,
%so by  Lemma \ref{lemma} the result follows. 
\end{proof}

\subsection{Proof of  Theorem \ref{main} }
\begin{proof}

By abuse of notation, if $b\in \HC^{\ZB}$ we will denote the corresponding object $b\in \HC^{\Bbbk}$, for any ring $\Bbbk$. By construction of the morphism spaces in $ \HC^{\Bbbk},$ double leaves are always a $\Bbbk$-basis of the Hom spaces between Bott-Samelson objects (see \cite{Li1}, \cite{Li2}, \cite{EW16}). So we have that 
\begin{equation}\label{isohom}
 \mathrm{Hom}_{ \HC^{\ZB}}(b, b')\otimes_{\ZB}\Bbbk \cong \mathrm{Hom}_{ \HC^{\Bbbk}}(b,b').
\end{equation}

Let $\FM_p$ be the finite field with $p$ elements, $\ZB_p$ the $p$-adic integers and $\QB_p$ the $p$-adic numbers. The isomorphism \eqref{isohom} gives sense to the following functors 
\begin{itemize}
\item $(-)\otimes_{\ZB_p}\QB_p:  \HC^{\ZB_p}\rightarrow  \HC^{\QB_p}$
\item $(-)\otimes_{\ZB_p}\FM_p:  \HC^{\ZB_p}\rightarrow  \HC^{\FM_p}$
\end{itemize}

\begin{nota}
For the rest of this proof, we will consider objects and morphisms in the Temperley-Lieb category (via the functor $\FS$) as if they were objects and morphisms in the Hecke category. For example, $^p\JW_n^{\QB}$ will be a morphism in $\HC^{\QB}.$
\end{nota}

We need to prove that $^p\JW_n^{\QB}$ seen as a morphism in  $\HC^{\QB_p}$ can be lifted to $ \HC^{\ZB_p}$  along the  functor $(-)\otimes_{\ZB_p}\QB_p$. We will prove it by induction on the number of non-zero coefficients in the $p$-adic expansion of $n+1$.

 If $n$ is a $p$-Adam, then $^p\JW_n^{\QB}=\JW_n$ and $n=jp^i-1$ with $0<j<p$ and $i\in \mathbb{N}$.
When specializing $\delta$ to $2$ the quantum binomial coefficients become ordinary binomial coefficients. Thus, if one applies  \cite[Theorem A.2]{EL} to the $p$-adic integers, one has  that the Jones-Wenzl projector is defined over  $\ZB_p$ if and only if for all $k<n$ the prime number $p$ does not divide the binomial coefficient ${n\choose k}$.

A consequence of Lucas's theorem is that a binomial coefficient ${n\choose k}$ is divisible by a prime number  $p$ if and only if at least one of the base $p$ digits of $k$ is greater than the corresponding digit of $n$. But if $n$ is a $p$-Adam this never happens because the $p$-adic expansion of $n$ is $(p-1)+(p-1)p+\cdots +(p-1)p^{i-1}+(j-1)p^i.$ Thus we conclude that  $\JW_n$ can be lifted to  $\ZB_p.$

Now we suppose that $^p\JW_{f[n]}^{\QB}$ can be  lifted to $ \HC^{\ZB_p}$. We will prove that $^p\JW_{n}^{\QB}$ can also be lifted. Let us say that $^p\JW_{f[n]}^{\ZB_p}$ is this lifting and $^p\JW_{f[n]}^{\FM_p}$  is the image of this morphism under the functor $(-)\otimes_{\ZB_p}\FM_p$.
As for any $b,b'$ objects of ${\HC^{\ZB}}$ we have
$$\mathrm{dim}(\mathrm{Hom}_{\HC^{\FM_p}}(b,b')) =\mathrm{rk}(\mathrm{Hom}_{\HC^{\ZB_p}}(b,b')) = \mathrm{dim}(\mathrm{Hom}_{\HC^{\QB_p}}(b,b')), $$
by  the formula for $\mathrm{ch}$ given in \cite[Definition 6.23]{EW16}, we obtain
at the decategorified level  $$\mathrm{ch}\left(\left\langle(f[n], {}^p\JW_{f[n]}^{\FM_p})\right\rangle \right)= \mathrm{ch}\left(\left\langle (f[n], {}^p\JW_{f[n]}^{\QB})\right\rangle \right)= \ ^pb_{\ul{f[n]+1}}.$$

So $(f[n],{}^p\JW_{f[n]}^{\FM_p})$ is isomorphic to the indecomposable object corresponding to the unique word of length $f[n]+1$  starting with $s$. In formulas  $$\left(f[n],{}^p\JW_{f[n]}^{\FM_p}\right)\cong \ ^pB_{\ul{f[n]+1}}\in \HC^{\FM_p}.$$ Recall that $m=n-f[n].$ The indecomposable object $ ^pB_{\ul{n+1}}$ is  a direct summand of $(n,{}^p\JW_{f[n]}^{\FM_p}\otimes \mathrm{id}_m)$. Let $\pi_{\FM_p}\in \mathrm{End}((n,{}^p\JW_{f[n]}^{\FM_p}\otimes \mathrm{id}_m))$ be the corresponding projector. Since $\mathrm{End}((n,{}^p\JW_{f[n]}^{\FM_p}\otimes \mathrm{id}_m))$ is a finitely generated $\ZB_p$-module, we can use idempotent lifting techniques for complete local rings (see \cite[Proposition 21.34 (1)]{Lam13}) and find
an  idempotent $\pi_{\ZB_p}\in \mathrm{End}((n,{}^p\JW_{f[n]}^{\FM_p}\otimes \mathrm{id}_m))$ mapping to  $\pi_{\FM_p}$.

By applying the corresponding functor one obtains an idempotent $$\pi_{\ZB_p}\otimes_{\ZB_p}\QB_p  \in \mathrm{End}\left((n,{}^p\JW_{f[n]}^{\FM_p}\otimes \mathrm{id}_m)\right)$$
that   decategorifies to $^pb_{\ul{n+1}}$ (just like $\pi_{\mathbb{Z}_p}$ and $\pi_{\mathbb{F}_p}$). But we have seen that $ {}^p\JW_{n}^{\QB}$ is the unique idempotent in 
 the endomorphism ring of $(n,{}^p\JW_{f[n]}^{\FM_p}\otimes \mathrm{id}_m)$ whose image categorifies $ {}^p b_{\ul{n+1}}$. Thus, $\pi_{\ZB_p}\otimes_{\ZB_p}\QB_p =\ {}^p\JW_{n}^{\QB}.$ This implies that ${}^p\JW_{n}^{\QB}$ can be lifted to $\pi_{\ZB_p}$ in $\HC^{\ZB_p}. $

\end{proof}

%%%%%% Bibliography
\bibliographystyle{alpha}
\bibliography{biblio}
%%%%%% Bibliography
\end{document}